\documentclass[preprint,12pt]{elsarticle}

\usepackage{amssymb,amsmath,amsthm,graphicx}%yhmath,
\usepackage{enumerate}
\usepackage{float}
\usepackage{stmaryrd,mathbbol,mathrsfs} 
\usepackage{multicol}
\usepackage{float}
\usepackage{stackrel}
\usepackage{color}
\usepackage{tikz}
\usetikzlibrary{matrix,arrows}
\usepackage{capt-of}
\usepackage{array}
\newtheorem{theorem}{Theorem}[section]
\newtheorem*{theorem*}{Theorem}
\newtheorem{lemma}[theorem]{Lemma}

\newtheorem{corollary}[theorem]{Corollary}

\theoremstyle{remark}

\theoremstyle{definition}
\newtheorem{example}[theorem]{Example}
\newtheorem{definition}[theorem]{Definition}

\makeatletter
\newcommand{\Spvek}[2][r]{%
  \gdef\@VORNE{1}
  \left(\hskip-\arraycolsep%
    \begin{array}{#1}\vekSp@lten{#2}\end{array}%
  \hskip-\arraycolsep\right)}

\def\vekSp@lten#1{\xvekSp@lten#1;vekL@stLine;}
\def\vekL@stLine{vekL@stLine}
\def\xvekSp@lten#1;{\def\temp{#1}%
  \ifx\temp\vekL@stLine
  \else
    \ifnum\@VORNE=1\gdef\@VORNE{0}
    \else\@arraycr\fi%
    #1%
    \expandafter\xvekSp@lten
  \fi}
\makeatother
\journal{}

\makeatletter
\def\ps@pprintTitle{%
 \let\@oddhead\@empty
 \let\@evenhead\@empty
 \def\@oddfoot{}%
 \let\@evenfoot\@oddfoot}
\makeatother

\begin{document}

\begin{frontmatter}

%% Title, authors and addresses

%% use the tnoteref command within \title for footnotes;
%% use the tnotetext command for the associated footnote;
%% use the fnref command within \author or \address for footnotes;
%% use the fntext command for the associated footnote;
%% use the corref command within \author for corresponding author footnotes;
%% use the cortext command for the associated footnote;
%% use the ead command for the email address,
%% and the form \ead[url] for the home page:
%%
%% \title{Title\tnoteref{label1}}
%% \tnotetext[label1]{}
%% \author{Name\corref{cor1}\fnref{label2}}
%% \ead{email address}
%% \ead[url]{home page}
%% \fntext[label2]{}
%% \cortext[cor1]{}
%% \address{Address\fnref{label3}}
%% \fntext[label3]{}

\title{On lattices with a smallest set of aggregation functions\footnote{Preprint of an article published by Elsevier in the Information Sciences 325 (2015), 316-323. It is available online at: \newline www.sciencedirect.com/science/article/pii/S0020025515005277}}

%% use optional labels to link authors explicitly to addresses:
%% \author[label1,label2]{}
%% \address[label1]{}
%% \address[label2]{}

\author{Radom\'ir Hala\v{s}}
\address{Palack\'y University Olomouc, Faculty of Science, Department of Algebra and Geometry, 17. listopadu 12, 771 46 Olomouc, Czech Republic}
\ead{radomir.halas@upol.cz}

\author{Jozef P\'ocs}
\address{Palack\'y University Olomouc, Faculty of Science, Department of Algebra and Geometry, 17. listopadu 12, 771 46 Olomouc, Czech Republic\\ 
and\\
Mathematical Institute, Slovak Academy of Sciences,\\ Gre\v s\'akova 6, 040 01 Ko\v sice, Slovakia}
\ead{pocs@saske.sk}

\begin{abstract}
Given a bounded lattice $L$ with bounds $0$ and $1$, it is well known that the set $\mathsf{Pol}_{0,1}(L)$ of all $0,1$-preserving polynomials of $L$ forms a natural subclass of the set $\mathsf{C}(L)$ of aggregation functions on $L$. The main aim of this paper is to characterize all finite lattices $L$ for which these two classes coincide, i.e. when the set $\mathsf{C}(L)$ is as small as possible. These lattices are shown to be completely determined by their tolerances, also several sufficient purely lattice-theoretical conditions are presented. In particular, all simple relatively complemented lattices or simple lattices for which the join (meet) of atoms (coatoms) is $1$ ($0$) are of this kind.
\end{abstract}

\begin{keyword}
polynomial\sep aggregation function\sep lattice \sep tolerance.
%% keywords here, in the form: keyword \sep keyword

%% PACS codes here, in the form: \PACS code \sep code
\MSC 06B99  
%% MSC codes here, in the form: \MSC code \sep code
%% or \MSC[2008] code \sep code (2000 is the default)

\end{keyword}

\end{frontmatter}
%%
%% Start line numbering here if you want
%%
% \linenumbers

\section{Introduction}

The problem of merging certain (usually numerical) data in a single output is one of the central problems of applied mathematics.  
Its mathematical theory is based on the notion of an aggregation function describing the process of merging. 

Aggregation functions can be found in many different branches of science, perhaps the most widely used one in all experimental sciences is the arithmetic mean. In fact, aggregation functions appear in a pure mathematics (functional equations, theory of means and averages, measure and integration theory), in applied mathematics (probability, statistics, decision theory), computer and engineering sciences (artificial intelligence, operation research, data fusion, automatic reasoning etc.). Let us mention that aggregation functions are not only used in natural sciences, quite recently they were successfully applied also in social sciences, economy, life sciences and other branches of research.

The central idea behind the process of aggregation is that it should somehow represent the ``synthesis" of input data, consequently the aggregation functions cannot be arbitrary and have to satisfy some natural minimal conditions: the output value should lie in the same domain as the input ones, and, additionally, the boundary values should be preserved.
The second natural widely accepted condition is nondecreasing monotonicity of the aggregation function, meaning that the increase of any of the input values should reflect this increase, or at worst, stay constant.

In case when the input (and, consequently, the output) values of 
these functions lie in a nonempty bounded real interval $\mathbb I=[a,b]$, the formal definition is as follows:
\noindent
an ($n$-ary) aggregation function on $\mathbb I^n$ is a function $A: \mathbb I^n\mapsto \mathbb I$ that 
\begin{itemize}
\item [(i)] is nondecreasing (in each variable)
\item [(ii)] fulfills the boundary conditions
\begin{equation}\label{eq1}
A(a,\dots,a)=a \quad {\rm and} \quad A(b,\dots,b)=b.
\end{equation}
\end{itemize}

The integer $n$ represents the arity of the aggregation function. Let us note that sometimes an additional 
condition for an aggregation function is required. Namely, since for $n=1$ there is nothing to aggregate, it is 
quite natural to ask $A(x)=x$. In our case this restriction is not considered.
For details we refer the reader e.g. to the comprehensive 
monographs \cite{GMRP}, \cite{CMM} or \cite{BPC}.

As bounded real intervals can be viewed as (complete) lattices, the theory of aggregation functions can be easily 
transferred to bounded lattices or even to bounded posets. The study of aggregation 
functions on lattices is a quite new quickly developing topic with possible applications in many areas of research, cf. \cite{Mes1,Mes2,M2,Mes3,Mes4,MZ1,MZ2,M1} or \cite{Pal1,Pal2,Pal3} for results concerning various lattice-valued connectives.    
%Recall that a lattice is 
%an algebra $(L;\vee,\wedge)$, where $L$ is a nonempty set with two binary operations $\vee$ and $\wedge$ representing 
%suprema and infima. Let us mention that lattice theory is a very well established discipline of universal algebra, there are 
%several monographs on this topic, among them the most frequently used and quoted are the books by G. Gr\"atzer, \cite{G1,G2}.

%Modifying the definition \eqref{eq1} of an aggregation function on a bounded real interval, quite recently and naturally the notion of aggregation function has been enlarged to bounded lattices. In %contrast to real valued functions, not much is known on aggregation functions on lattices. We can mention e.g. the paper \cite{KM} devoted to a certain classification of aggreagation functions on %bounded partially ordered sets and lattices. 

Certainly, one of the central problems connected with aggregation functions is their construction. It is easy to see that for any bounded lattice $L$, the set $\mathsf{Pol}_{0,1}(L)$ of $0,1$-preserving polynomials of $L$ represents a natural subclass of the set $\mathsf{C}(L)$ of aggregation functions on $L$. This simple fact immediately leads to the following problem:
\begin{itemize}
\item Characterize lattices $L$, for which $\mathsf{C}(L)=\mathsf{Pol}_{0,1}(L),$
\end{itemize}
\noindent
i.e., lattices for which the set of aggregation functions $\mathsf{C}(L)$ is as small as possible. We call them {\it lattices with a smallest set of aggregation functions}. Through the paper we will denote the class of all finite lattices satisfying this property by $\mathsf{S}_{\mathsf{agg}}$.

Let us note that the problem concerning a characterization of the class $\mathsf{S}_{\mathsf{agg}}$ is closely related to that of polynomial representability of various types of functions, cf. \cite{KP}. The main advantage of our approach relies on the use of elementary techniques, it is shown that finite lattices belonging to the class $\mathsf{S}_{\mathsf{agg}}$ are completely determined by their tolerances. Also several sufficient purely lattice-theoretical conditions are presented. In particular, we prove that all simple relatively complemented lattices or simple lattices for which the join (meet) of atoms (coatoms) is $1$ ($0$) are of this kind.

The paper is organized as follows: first, we give an overview of some basic definitions and facts on lattices, polynomials and tolerances. In the last section we present a characterization of finite lattices having smallest sets of aggregation functions.

\section{Lattices, polynomials, tolerances}

To make the paper self-contained, we recall some necessary concepts from universal algebra and lattice theory, for more details we refer the reader to the comprehensive monographs cf. \cite{Burris} or \cite{McKenzie}. Let us mention that lattice theory is a very well established discipline of universal algebra, there are 
several monographs on this topic, among them the most frequently used are the books by G. Gr\"atzer, \cite{G1,G2}. 

First, we recall the definition of lattices as algebraic structures.
\begin{definition}
An algebraic structure $(L;\vee,\wedge)$ consisting of a nonempty set $L$ and two binary operations $\vee$, $\wedge$ on $L$ is called a \textit{lattice} if for all $a,b$ and $c$ in $L$ the following hold:
\begin{equation*} 
\begin{aligned}[c]
a&=a\vee a, \\
a\vee b&=b\vee a, \\
a\vee (b\vee c)&=(a\vee b)\vee c,\\
a&=a\vee (a\wedge b),
\end{aligned}
\quad\quad
\begin{aligned}[c]
a\wedge a&=a. \\
a\wedge b&=b\wedge a.\\
a\wedge (b\wedge c)&=(a\wedge b)\wedge c.\\
a\wedge (a\vee b)&=a.
\end{aligned}
\end{equation*}
\end{definition}

Let us note that a lattice can be equivalently characterized as a poset $(L,\leq)$ such that $\sup\{a,b\}$ and $\inf\{a,b\}$ exist for all $a,b\in L$. In this case two binary operations $\vee$ and $\wedge$ representing suprema and infima fulfill the algebraic definition of lattices. Conversely, given a lattice $(L;\vee,\wedge)$ one can define 
$$ a\leq b \quad \mbox{iff}\quad b=a\vee b \quad \mbox{iff}\quad a=a\wedge b,$$ obtaining a lattice partial order on $L$. To simplify expressions, we usually do not distinguish between the lattice and its support. 

%Recall that a {\it lattice} is 
%an algebra $(L;\vee,\wedge)$, where $L$ is a nonempty set (called {\it support}) with two binary operations $\vee$ and $\wedge$ representing 
%suprema and infima. More precisely, both $(L;\vee)$, $(L;\wedge)$ are commutative and idempotent semigroups connected by the absorption laws 
%$x\vee (x\wedge y)=x=x\wedge (x\vee y)$.

%To simplify expressions, we usually do not distinguish between the lattice and its support. 

By a {\it sublattice} of $L$ is meant a subset $B\subseteq L$ closed under suprema and infima, i.e. fulfilling 
the properties $a\vee b\in B$ and $a\wedge b\in B$ for all $a,b\in B$. Equivalently, $B$ is a sublattice iff $(B;\vee,\wedge)$ is a lattice.

Given a lattice $(L;\vee,\wedge)$, by its {\it direct square} we mean a lattice $(L^2;\vee,\wedge)$ with the support $L^2$ being the Cartesian square of $L$ and lattice operations defined component-wise, i.e. $(a,b)\vee (c,d):=(a\vee c,b\vee d)$ and $(a,b)\wedge (c,d):=(a\wedge c,b\wedge d)$ for all $a,b,c,d\in L$. A sublattice $B$ of $L^2$ is called {\it diagonal} whenever 
$id_{L}=\{(a,a)\in L^2;\,a\in L\}\subseteq B$.   

\begin{definition}
Let $L$ be a lattice. A binary relation $T$ is \textit{compatible} on the lattice $L$ if $(a,b),(c,d)\in T$ imply $(a\vee c,b\vee d)\in T$ and $(a\wedge c,b\wedge d)\in T$ for any $a,b,c,d\in L$. A {\it tolerance} on a lattice $L$ is any reflexive, symmetric and compatible binary relation on $L$. By a \textit{congruence} we understand any compatible equivalence on $L$. Finally, $L$ is called \textit{simple}, if any congruence on $L$ is either $id_L$ or $L^2$.
\end{definition}

%A {\it tolerance} $T$ on a lattice $L$ is any reflexive, symmetric and compatible relation on $L$, i.e. relation for which $(a,b),(c,d)\in T$ imply $(a\vee c,b\vee d)\in T$ and $(a\wedge c,b\wedge d)\in T$ for any $a,b,c,d\in L$.
Using lattice operations on the direct square $L^2$ of $L$, tolerances can be viewed by another equivalent way: these are exactly diagonal symmetric sublattices of $L^2$. Note that congruences on $L$ are just its transitive tolerances.

Clearly, with respect to set inclusion $id_L$ is the least, while $L^2$ is the greatest tolerance on $L$, and we have $\mathsf{Con}(L)\subseteq \mathsf{Tol}(L)$ where $\mathsf{Tol}(L)$ resp. $\mathsf{Con}(L)$ is the set of tolerances resp. congruences of $L$. 
%We call a lattice {\it simple} if it has trivial congruences only, i.e. when $\mathsf{Con}(L)=\{id_{L},L^2\}$.

\begin{definition}
Let $L$ be a lattice and $n\in \mathbb{N}\cup \{0\}$ be a non-negative integer. By an $n$-ary {\it polynomial} on the lattice $L$ we mean any function $p: L^n\to L$ defined inductively as follows:
\begin{enumerate}
\item[--] For each $i\in \{1,\dots,n\}$ the $i$-th projection $p(x_1,\dots,x_n)=x_i$ is a polynomial.
\item[--] Any constant function $p(x_1,\dots,x_n)=a$ for $a\in L$ is a polynomial.
\item[--] If $p_{1}(x_1,\dots,x_n)$ and $p_{2}(x_1,\dots,x_n)$ are polynomials, then so does the functions $p_{1}(x_1,\dots,x_n)\vee p_{2}(x_1,\dots,x_n)$ and $p_{1}(x_1,\dots,x_n)\wedge p_{2}(x_1,\dots,x_n)$.
\item[--] Any polynomial is obtained by finitely many of the preceding steps. 
\end{enumerate}
\end{definition}

Informally, lattice polynomials are functions obtained by composing variables and constant functions by using of lattice operations.
Note that polynomials defined in this way are called as weighted lattice polynomials in \cite{M}.
%By an $n$-ary {\it polynomial} ($n\in \mathbb{N}\cup \{0\}$) on a lattice $L$ is meant any function $p: L^n\to L$ obtained by composing functions using of lattice operations $\vee$ and $\wedge$ and constant $0$-ary operations $c_{a}:=a$, $a\in L$. More precisely, $n$-ary polynomials are defined inductively as follows\\
%- any variable $p(x_1,\dots,x_n)=x_i$, $i\in \{1,\dots,n\}$, is a polynomial\\
%\noindent
%- any constant function $p(x_1,\dots,x_n)=c_a$, $a\in L$, is a polynomial\\
%\noindent
%- if $p_{1}(x_1,\dots,x_n)$ and $p_{2}(x_1,\dots,x_n)$ are polynomials, then so does the functions $p_{1}(x_1,\dots,x_n)\vee p_{2}(x_1,\dots,x_n)$ and $p_{1}(x_1,\dots,x_n)\wedge p_{2}(x_1,\dots,x_n)$\\
%\noindent
%- any polynomial is obtained by finitely many of the preceding steps.

We denote by $\mathsf{Pol}_u(L)$ the set of {\it unary polynomials} (i.e., polynomials of arity $n=1$) on a lattice $L$, 
and by $\mathsf{Pol}_{0,1}(L)$ polynomials preserving the bounds $0$ and $1$, i.e. those fulfilling 
$p(0,\dots,0)=0$ and $p(1,\dots,1)=1$.

Further, for $x,y\in L, x\leq y$, let $[x,y]=\{z\in L;\, x\leq z\leq y\}$ be the {\it interval} in $L$ between $x$ and $y$. 

\vskip10pt

We present several basic properties of tolerances on lattices which will be used in the next section:  

\begin{lemma}\label{lemn3}
Let $L$ be a lattice, $T\in \mathsf{Tol}(L)$ and $a,b\in L$. Then 
\begin{enumerate}[{\rm (i)}]
\item $(a,b)\in T$ and $a\leq b$ imply $(x,y)\in T$ for all $x,y\in L$ with $a\leq x,y\leq b$,
\item $(a,b)\in T$ yields $[a\wedge b,a\vee b]^2\subseteq T$,
\item $(a,b)\in T$ implies $\big(f(a),f(b)\big)\in T$ for all $f\in\mathsf{Pol}_u(L)$,
\item $B\circ B^{-1}\in \mathsf{Tol}(L)$ for any diagonal sublattice $B$ of $L^2$.
\end{enumerate}
\end{lemma}

\begin{proof}
(i) Let $x,y\in[a,b]$ be arbitrary elements. Since $T$ is reflexive, it follows that $(x,x)\in T$ and compatibility of T with the lattice operations yields $(a\vee x, b\vee x)=(x,b)\in T$. Similarly $(y,b)\in T$ and due to symmetry and compatibility of $T$ we obtain $(x\wedge b,b\wedge y)=(x,y)\in T$.

(ii) In view of (i) it is sufficient to show that $(a\wedge b,a\vee b)\in T$. Again, $(a,b)\in T$ and $(b,b)\in T$ imply $(a\wedge b,b)\in T$ and, analogously, $(a\wedge b,a)\in T$. From this we obtain $\big((a\wedge b)\vee(a\wedge b),a\vee b\big)=(a\wedge b,a\vee b)\in T$.

(iii) We prove this claim by induction according to the number $m$ of the lattice operations involved in a polynomial. If $m=0$, then either $f(x)=c$ is a constant polynomial or $f(x)=x$ is the identity on $L$. Since $T$ is reflexive, in the first case we obtain $\big(f(a),f(b)\big)=(c,c)\in T$, while in the second case we have $\big(f(a),f(b)\big)=(a,b)\in T$. 

Further, assume that the assertion is valid for all polynomials involving fewer than $m$ lattice operations and let $f$ be a polynomial which contains precisely $m$ operations. Then $f$ can be expressed as $f(x)=f_1(x)\vee f_2(x)$ or $f(x)=f_1(x)\wedge f_2(x)$ for some polynomials $f_1,f_2$ containing fewer than $m$ operations. Using induction assumption, in both cases we obtain $\big(f(a),f(b)\big)\in T$.

(iv) Obviously, if $id_L\subseteq B$ then $B\circ B^{-1}$ is reflexive. Since $(B\circ B^{-1})^{-1}=B\circ B^{-1}$, it is symmetric as well. In order to prove the compatibility of $B\circ B^{-1}$, assume that $(a,b),(c,d)\in B\circ B^{-1}$, i.e. $(a,x)\in B, (x,b)\in B^{-1}$ and $(c,y)\in B, (y,d)\in B^{-1}$ for some $x,y\in L$. Consequently, $(b,x),(d,y)\in B$, and as $B$ is a sublattice of 
$L^2$, we conclude $(a\wedge c,x\wedge y),(b\wedge d,x\wedge y)\in B$. Thus $(x\wedge y,b\wedge d)\in B^{-1}$ and $(a\wedge c,b\wedge d)\in B\circ B^{-1}$.  

The compatibility of $B\circ B^{-1}$ with respect to join operation can be shown similarly.   

\end{proof}

For more details on tolerances of lattices we refer the reader to the comprehensive monograph \cite{Ch}.

%To simplify expressions, for any $n$-ary function $f: L^n\to L$ on a lattice $L$ and $\mathbf{x}=(x_1,\dots,x_n)\in L^n$ we put %$f({\mathbf x}):=f(x_1,\dots,x_n)$. 

% Usually we do not distinguish between the algebra $\mathbf A$ and its support $A$.

\section{Lattices with a smallest set of aggregation functions}

In the sequel we will assume that all lattices are finite.
To simplify expressions, for any $n$-ary function $f: L^n\to L$ on a lattice $L$ and $\mathbf{x}=(x_1,\dots,x_n)\in L^n$ we put $f({\mathbf x}):=f(x_1,\dots,x_n)$. 
Obviously, the framework of aggregation functions can be modified by considering functions on any closed real interval, and clearly to any partially ordered structure with bounds (see e.g. \cite{KM}):

\begin{definition}
Let $(P,\leq,0,1)$ be a bounded partially ordered set (poset), let $n\in \mathbb N$. A mapping $A: P^n\to P$ is called 
an ($n$-ary) aggregation function on $P$ if it is nondecreasing, i.e. for any $\mathbf{x}, \mathbf{y}\in P^n$: 
$$A(\mathbf{x})\leq A(\mathbf{y})\,\, \text{ whenever }\,\, \mathbf{x}\leq \mathbf{y},$$
and it satisfies boundary conditions 
$$A(0,\dots,0)=0 \quad\text{ and }\quad A(1,\dots,1)=1.$$
\end{definition}  

For a more detailed discussion on aggregation functions on posets or lattices we recommend the paper by Demirci \cite{Dem}. 
Special types of aggregation functions on posets, especially triangular norms or conorms, are studied e.g. in \cite{DeBM,M2,M1}. 
It is easy to see that considering $P=[0,1]$ to be the standard interval of reals with the usual ordering, we obtain the   
classical definition of an aggregation function. 

A particular example of a ternary aggregation function on a $3$-element chain is schematically shown on Figure \ref{fig3}.

\begin{figure}
\begin{center}
\includegraphics[scale=1]{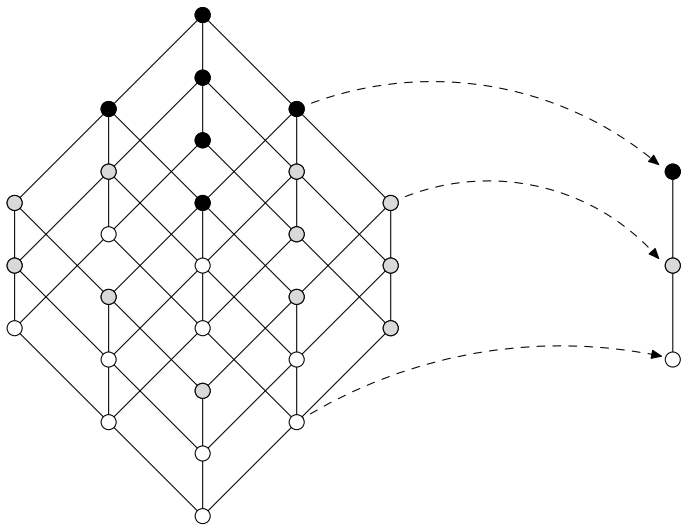}
\caption{A ternary aggregation function on a three-element chain}
\label{fig3}
\end{center}
\end{figure}

Denote by $\mathsf{C}(L)$ the set of aggregation functions on a lattice $L$. One can easily see that the set $\mathsf{Pol}_{0,1}(L)$ of $0,1$-preserving polynomials on $L$ is included in $\mathsf{C}(L)$. 

The important role concerning the solution of our problem to characterize the class $\mathsf{S}_{\mathsf{agg}}$ is played by the following unary aggregation functions on a lattice $L$: 

for any $a\in L$ we define $\chi_a\colon L\to L$ by
\begin{equation}\label{e1}
\chi_a(x)=
\begin{cases}&1, \text{ if }x\geq a, x\neq 0; \\
             &0, \text{ otherwise.}\\   
\end{cases}
\end{equation} 

Obviously, $\chi_a$ is an aggregation function for all $a\in L$. Moreover, it represents a characteristic function of the principal filter $F(a)=\{x\in L:x\geq a\}$ generated by $a$, provided $a\neq 0$. 

In what follows we will show that lattices with a smallest set of aggregation functions are completely characterized by their tolerances. 

Recall that an element $a\in L$ is \textit{join-irreducible} if $a\neq 0$ and $a=b\vee c$ yields  $a=b$ or $a=c$. For a finite lattice $L$, let $\mathsf{J}(L)$ denotes the set of all join-irreducible elements. Obviously, for any $0\neq a\in L$, $a=\bigvee X$ holds for some $X\subseteq \mathsf{J}(L)$. Let us notice that if an element $a\in L$ covers more than one element then it is not join irreducible. Consequently, for a finite lattice $L$, an element $a\in L$ is join-irreducible if and only if $a$ covers a unique element of $L$.

Further, an element $a$ of a lattice $L$ is called an {\it atom} if $a$ covers $0$, i.e. if $a>0$ and there is no 
element $b\in L$ with $a>b>0$. The set of all atoms of $L$ will be denoted by $\mathsf{At}(L)$. Dually, the elements 
of $L$ covered by $1$ are called its {\it coatoms}.

The following lemma shows that the aggregation functions $\chi_a$ for $a\in \mathsf{J}(L)$ play a crucial role in description of lattices from the class $\mathsf{S}_{\mathsf{agg}}$. 

\begin{lemma}\label{lemn1}
Let $L$ be a finite lattice. Then $\mathsf{C}(L)=\mathsf{Pol}_{0,1}(L)$ if and only if $\chi_a$ is a polynomial for each $a\in \mathsf{J}(L)$.
\end{lemma}

\begin{proof}
Since any $\chi_a$ for $a\in \mathsf{J}(L)$ is an aggregation function, $\mathsf{C}(L)=\mathsf{Pol}_{0,1}(L)$ obviously implies that $\chi_a$ is a polynomial.

Conversely, assume that for every $a\in \mathsf{J}(L)$ the aggregation function $\chi_a$ is a polynomial. We show that under this assumption any aggregation function can be represented as a polynomial. 

First, we show that $\chi_c$ is a polynomial for all $c\in L$. Assume that $c\neq 0$ and $c=\bigvee X$ for some subset $\emptyset\neq X\subseteq \mathsf{J}(L)$. Given an arbitrary element $x\in L$ we obtain $\bigwedge_{a\in X}\chi_a(x)=1$ if and only if $\chi_a(x)=1$ for all $a\in X$, which is equivalent to $x\geq a$ for all $a\in X$. Evidently, this condition holds if and only if $x\geq \bigvee X=c$, yielding  $\chi_c(x)=\bigwedge_{a\in X}\chi_a(x)$. Hence the aggregation function $\chi_c$ is a polynomial. Moreover $\chi_0(x)=\bigvee_{a\in \mathsf{At}(L)} \chi_a(x)$. 

Further, let $n\geq 1$ be a positive integer. Denote by $L^n_0$ the set of all non-zero elements of $L^n$ and for $\mathbf{a}=(a_1,\dots,a_n)\in L^n_0$ we denote by 
$I_a=\{1\leq i\leq n: a_i\neq 0\}$ the set of all non-zero indexes. For an aggregation function $f\colon L^n\to L$ we put 

$$p(x_1,\dots,x_n)=\bigvee_{\mathbf{a}\in L^n_0} \big( f(\mathbf{a})\wedge\bigwedge_{i\in I_a} \chi_{a_i}(x_i)\big).$$

Obviously, $p\colon L^n\to L$ is a polynomial and we will show that $f(\mathbf{x})=p(\mathbf{x})$ for all $\mathbf{x}\in L^n$. For $\mathbf{x}=(0,\dots,0)$ we obtain

$$ p(0,\dots,0)=\bigvee_{\mathbf{a}\in L^n_0} \big( f(\mathbf{a})\wedge\bigwedge_{i\in I_a} \chi_{a_i}(0)\big)=\bigvee_{\mathbf{a}\in L^n_0} 0=0=f(0,\dots,0),$$
since $\chi_c$ for each $c\in L$ as well as $f$ satisfy the boundary condition for aggregation functions. For an $n$-tuple $\mathbf{x}=(x_1,\dots,x_n)\neq (0,\dots,0)$ we have
$$ p(\mathbf{x})=\bigvee_{\substack{\mathbf{a}\in L^n_{0}\\ \mathbf{a}\leq \mathbf{x}}} \big( f(\mathbf{a})\wedge\bigwedge_{i\in I_a} \chi_{a_i}(x_i)\big) \ \vee\  \bigvee_{\substack{\mathbf{a}\in L^n_{0}\\ \mathbf{a}\nleq \mathbf{x}}} \big( f(\mathbf{a})\wedge\bigwedge_{i\in I_a} \chi_{a_i}(x_i)\big).$$ 
As $\mathbf{a}\leq \mathbf{x}$ if and only if $a_i\leq x_i$ for all $i\in I_a$, we obtain $\bigwedge_{i\in I_a} \chi_{a_i}(x_i)=1$ provided $\mathbf{a}\leq \mathbf{x}$, while $\bigwedge_{i\in I_a} \chi_{a_i}(x_i)=0$ if $\mathbf{a}\nleq\mathbf{x}$. This yields 
$$ p(\mathbf{x})=\bigvee_{\substack{\mathbf{a}\in L^n_{0}\\ \mathbf{a}\leq \mathbf{x}}} \big( f(\mathbf{a})\wedge 1\big) \ \vee\  \bigvee_{\substack{\mathbf{a}\in L^n_{0}\\ \mathbf{a}\nleq \mathbf{x}}} \big( f(\mathbf{a})\wedge 0\big)=\bigvee_{\substack{\mathbf{a}\in L^n_{0}\\ \mathbf{a}\leq \mathbf{x}}} f(\mathbf{a}).$$ 
The monotonicity of $f$ implies $f(\mathbf{a})\leq f(\mathbf{x})$ for all $\mathbf{a}\in L^n_0$ with $\mathbf{a}\leq \mathbf{x}$. Finally,  we obtain 
$$p(\mathbf{x})=\bigvee_{\substack{\mathbf{a}\in L^n_{0}\\ \mathbf{a}\leq \mathbf{x}}} f(\mathbf{a})=f(\mathbf{x}),$$ which completes the proof.

%Conversely, assume that for every $a\in \mathrm{J}(L)$ the aggregation function $\chi_a$ is a polynomial. Recall that the aggregation clone $\mathcal{C}_L$ can be generated by the lattice operations $\vee$ and $\wedge$ together with the family of functions $\chi_c$ and $\oplus_c$ for $c\in L$. Hence, in order to prove $\mathcal{C}_L=\mathsf{Pol}_{0,1}(L)$ it is sufficient to show that each of the function from the generating family is a polynomial. Assume that $c=a\vee b$ in $L$. Then $\chi_a(x)\wedge \chi_b(x)=1$ if and only if $x\geq a$ and $x\geq b$, which yields $\chi_c=\chi_a\wedge \chi_b$. Moreover $\chi_0(x)=\bigvee_{a\in \mathsf{At}(L)} \chi_a(x)$, where $\mathsf{At}(L)$ denotes the set of all atoms of $L$. Since any atom is join-irreducible and each element is a supremum of join-irreducible elements, it follows that $\chi_c$ is a polynomial for all $c\in L$. Finally, given an element $c\in L$,  
%$$x\oplus_c y=\big(\chi_0(x\vee y)\wedge c\big)\vee \chi_1(x\wedge y)$$ 
%holds for all pairs $x,y\in L$. Hence, any function from the generating family can be expressed as a polynomial and we obtain $\mathcal{C}_L=\mathsf{Pol}_{0,1}(L)$.   
\end{proof}

Based on the previous lemma, we present two nontrivial examples of lattices from $\mathsf{S}_{\mathsf{agg}}$.

\begin{example}

For a positive integer $n\geq 3$ consider the lattice $M_n$, consisting of $n$ mutually incomparable elements $a_1,\dots, a_n$, together with the universal bounds $0$ and $1$, cf. Figure \ref{fig1}.
Note that $M_n$ is a horizontal sum of $n$ three-element chains, see e.g. \cite{CHK}. 

\begin{figure}
\begin{center}
\includegraphics{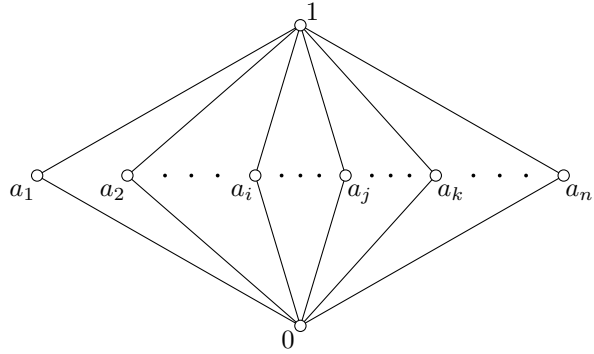}
\caption{The lattice $M_n$}
\label{fig1}
\end{center}
\end{figure}

Given an element $a_i$, let $a_j,a_k$ with $j\neq i\neq k$ be two different elements. Then it can be easily verified that  
$$\chi_{a_i}(x)=\big((x\wedge a_i)\vee a_j\big)\wedge\big((x\wedge a_i)\vee a_k\big).$$
Consequently, $\mathsf{C}(M_n)=\mathsf{Pol}_{0,1}(M_n)$ according to Lemma \ref{lemn1}. 
\end{example}

\begin{example}

Let $L$ be the lattice depicted in Figure \ref{fig2}.
\begin{figure}
\begin{center}
\includegraphics[scale=1.2]{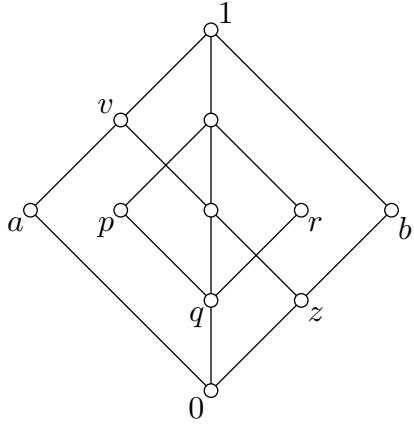}
\caption{The non-complemented, non-modular lattice $L$, satisfying $\mathsf{C}(L)=\mathsf{Pol}_{0,1}(L)$}
\label{fig2}
\end{center}
\end{figure}
We show that all the functions $\chi_c$ for $c\in \mathrm{J}(L)$ are polynomials. First, we find a polynomial for the function $\chi_a$. It can be easily seen that 
$$\chi_a(x)=\left[(x\wedge a)\vee r\right] \wedge \left[(x\wedge a)\vee b\right].$$
The following expressions show that all the functions $\chi_c$ for $c\in \mathsf{J}(L)=\{a,p,q,r,z,b\}$ are polynomials as well:

$\chi_b(x)=\chi_a\big((x\wedge b)\vee q\big)$, $ \chi_1(x)=\chi_a(x) \wedge \chi_b(x)$, $ \chi_q(x)=\chi_1\big((x\wedge q) \vee b \big)$, 

$ \chi_p(x)=\chi_1\big((x\wedge p) \vee v \big)$, $ \chi_r(x)=\chi_1\big((x\wedge r) \vee v \big)$, $ \chi_v(x)=\chi_a(x)\wedge \chi_q(x)$,

$ \chi_z(x)=\chi_v\big((x\wedge z) \vee a \big)$.
%$\chi_0(x)=\chi_a(x)\vee \chi_q(x)\vee \chi_z(x)$.

\end{example}

We have transformed our problem to deciding whether or not the functions $\chi_a$ for $a\in \mathsf{J}(L)$ are polynomials. The following lemma gives a useful sufficient condition for a function of this form to be a polynomial on $L$.  

\begin{lemma}\label{lemn2}
Let $L$ be a finite lattice, $a\in \mathsf{J}(L)$ be a join-irreducible element and $b$ be the unique element covered by $a$. If there is a polynomial $f\in \mathsf{Pol}_u(L)$ satisfying $f(b)=0$ and $f(a)=1$ then $\chi_a$ is a polynomial as well.
\end{lemma}

\begin{proof}
Since $x\wedge a=a$ if and only if $x\geq a$ and $b$ is the greatest element of the set $\{x\in L: x<a\}$, it follows that $x\wedge a\leq b$ for all $x\ngeq a$. Consequently, using monotonicity of $f$ we obtain for all $x\in L$ 
$$
\chi_a(x)=f(x\wedge a)=\begin{cases}
f(a)=1,\ \text{if}\ x\geq a,\\
f(x\wedge a)\leq f(b)=0,\ \text{if}\ x\ngeq a.
\end{cases}
$$
\end{proof}

Now we are ready to show that lattices from the class $\mathsf{S}_{\mathsf{agg}}$ necessarily have only trivial tolerances. 

\begin{lemma}\label{lemn4}
Let $L$ be a lattice. Then $\mathsf{C}(L)=\mathsf{Pol}_{0,1}(L)$ yields $\mathsf{Tol}(L)=\{id_L,L^2\}$.
\end{lemma}

\begin{proof}
Assume that $\mathsf{C}(L)=\mathsf{Pol}_{0,1}(L)$ and let $T\in\mathsf{Tol}(L)$ be a tolerance relation satisfying $T\neq id_L$. Thus there is a pair of elements $(a,b)\in T$ such that $a<b$. From the assumption $\mathsf{C}(L)=\mathsf{Pol}_{0,1}(L)$ and $\chi_b\in \mathsf{C}(L)$ we conclude $\chi_b\in {Pol}_u(L)$. By Lemma \ref{lemn3} (iii), $\chi_b$ preserves $T$, and as $(a,b)\in T$, we conclude $\big(\chi_b(a),\chi_b(b)\big)=(0,1)\in T$. Hence, by Lemma \ref{lemn3} (ii) we obtain $[0,1]^2\subseteq T$, and $T=L^2$. This shows $\mathsf{Tol}(L)=\{id_L,L^2\}$, completing the proof.  
\end{proof}

The following lemma shows that the previous necessary condition concerning tolerances of $L$ is also sufficient.

\begin{lemma}
For a finite lattice $L$, the condition $\mathsf{Tol}(L)=\{id_L,L^2\}$ implies $\mathsf{C}(L)=\mathsf{Pol}_{0,1}(L)$.
\end{lemma}

\begin{proof}
By Lemma \ref{lemn1} it is sufficient to show that $\chi_a$ is a polynomial for each $a\in\mathsf{J}(L)$. Hence, let $a\in \mathsf{J}(L)$ be an arbitrary join-irreducible element and $b$ be the unique element covered by $a$. 

First, we show that $B=\big\{\big(p(b),p(a)\big):p\in\mathsf{Pol}_u(L)\big\}$ is a diagonal sublattice of $L^2$. Suppose $(x_1,y_1),(x_2,y_2)\in B$, i.e., there are two polynomials $p_1,p_2\in\mathsf{Pol}_u(L)$ with $p_1(b)=x_1$, $p_1(a)=y_1$ and $p_2(b)=x_2$, $p_2(a)=y_2$. Then 
$$(x_1,y_1)\vee (x_2,y_2)=(x_1\vee x_2,y_1\vee y_2)=\big(p_1(b)\vee p_2(b),p_1(a)\vee p_2(a)\big)$$ and 
$$(x_1,y_1)\wedge (x_2,y_2)=(x_1\wedge x_2,y_1\wedge y_2)=\big(p_1(b)\wedge p_2(b),p_1(a)\wedge p_2(a)\big).$$ Since $p_1\vee p_2$ as well as $p_1\wedge p_2$ are polynomials, the set $B$ is closed under the lattice operations in $L^2$. The inclusion $id_L\subseteq B$ follows from the fact that each constant function is a polynomial.

According to Lemma \ref{lemn3} (iv), the composition $B\circ B^{-1}$ is a tolerance relation of the lattice $L$. Moreover, we have $(b,a)\in B\circ B^{-1}$ since the identical function is polynomial and $B^{-1}$ is reflexive. Consequently, $B\circ B^{-1}\neq id_L$ and the assumption $\mathsf{Tol}(L)=\{id_L,L^2\}$ yields $B\circ B^{-1}=L^2$. Then $(0,1)\in B\circ B^{-1}$, which implies $(0,x)\in B$ and $(x,1)\in B^{-1}$ for some $x\in L$. Due to the definition of $B$, there are $p,q\in\mathsf{Pol}_u(L)$ such that $(0,x)=\big(p(b),p(a)\big)$ and $(1,x)=\big(q(b),q(a)\big)$. From this fact and monotonicity of $q$ we obtain $1= q(b)\leq q(a)=x$. This yields $p(b)=0$ together with $p(a)=x=1$. Finally, from Lemma \ref{lemn2} we obtain that $\chi_a(x)=p(x\wedge a)$ is a polynomial, which completes the proof.
\end{proof}

Consequently, we obtain the main result of the paper describing lattices having smallest sets of aggregation functions by means of their tolerances:

\begin{theorem}
Let $L$ be a finite lattice. Then $L\in \mathsf{S}_{\mathsf{agg}}$ if and only if $\mathsf{Tol}(L)=\{id_L,L^2\}$.
\end{theorem}  

We have seen that lattices having a smallest set of aggregation functions are necessarily simple. 
As the following example shows, there are simple lattices with non-trivial tolerance relations.

\begin{example}
Consider the lattice $L$ from Figure \ref{fig4}. It is a well-known fact that the lattice $M_3$ is simple, the same is easily seen for $L$ as a ``gluing" of two copies of $M_3$ and having two elements in common.   

Further, consider the relation $T=T_1^2\cup T_2^2$. Obviously, $T$ is reflexive and symmetric. Moreover, $T_1$ and $T_2$ are sublattices of $L$. Consequently, $x\vee y, x\wedge y \in T_i$ if $x,y\in T_i$ for $i\in \{1,2\}$. Further, $x\vee y\in T_2$ if $x\in T_1$ and $y\in T_2$,  $x\wedge y\in T_1$ whenever $x\in T_1$ and $y\in T_2$. 

Now, let $(a,b), (c,d)\in T$ be arbitrary elements. As $T=T_1^2\cup T_2^2$, first assume that $(a,b)\in T_1^2$. Then $(a\vee c,b\vee d)\in T_1^2$ if $(c,d)\in T_1^2$, while $(a\vee c,b\vee d)\in T_2^2$ if $(c,d)\in T_2^2$. If $(a,b)\in T_2^2$, then always $(a\vee c,b\vee d)\in T_2^2$. The compatibility of $T$ with respect to the meet operation can be proved analogously. Hence, $T$ is a non-trivial tolerance relation on $L$.

\end{example}
 
\begin{figure}
\begin{center}
\includegraphics[scale=1]{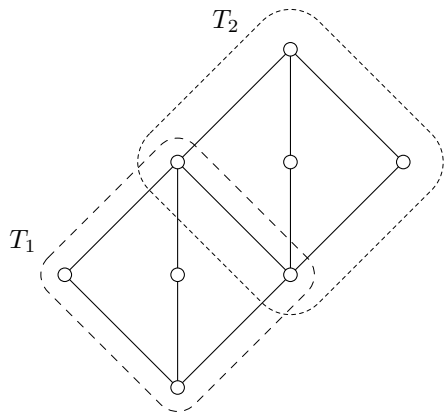}
\caption{A simple lattice having a non-trivial tolerance}
\label{fig4}
\end{center}
\end{figure}

In what follows we will present several purely lattice-theoretical conditions for lattices which guarantee the triviality of their tolerances. Recall that a bounded lattice $L$ is called {\it complemented} if any element $x\in L$ has a complement $y\in L$, i.e. we have $x\wedge y=0$ and $x\vee y=1$. A lattice $L$ is {\it relatively complemented} whenever each of its 
intervals $[a,b]$ for $a\leq b$ is a complemented lattice. 

\begin{theorem}
Let $L$ be finite, bounded, simple, relatively complemented lattice. Then $L\in \mathsf{S}_{\mathsf{agg}}$.
\end{theorem}
\begin{proof}
We have to prove that $\mathsf{Tol}(L)=\{id_L,L^2\}$. As $L$ is assumed to be simple, it is enough to show that $\mathsf{Tol}(L)=\mathsf{Con}(L)$, i.e. that every tolerance $T$ on $L$ is transitive: 
for all $a,b,c\in L$, $(a,b),(b,c)\in T$ implies $(a,c)\in T$. 

First, we prove that if transivity of $T$ holds for any triple $a,b,c\in L$ with $a\geq b\geq c$, then so does for any triple $a,b,c\in L$. Indeed, let $(a,b),(b,c)\in T$ for $a,b,c\in L$. Applying compatibility of $T$, we obtain $(a,a\wedge b)\in T$ and hence 
$(a\wedge b,a\wedge b\wedge c)=(a\wedge b,(a\wedge b)\wedge c)\in T$. Clearly, $a\geq a\wedge b\geq a\wedge b\wedge c$, thus by our assumption we conclude $(a,a\wedge b\wedge c)\in T$. Similarly, exchanging the elements $a$ and $c$ and using the symmetry 
of $T$ we obtain $(a\wedge b\wedge c,c)\in T$ as well. Further, compatibility of $T$ yields 
$$(a,c)=(a\vee (a\wedge b\wedge c),(a\wedge b\wedge c)\vee c)\in T.$$

Now, we are ready to prove that any tolerance $T$ on $L$ is transitive: for this assume $(a,b),(b,c)\in T$ for some $a,b,c\in L$ with $a\geq b\geq c$. As $L$ is relatively complemented, there exists a complement $d$ of $b$ in the interval $[c,a]$, i.e. $b\vee d=a$ and $b\wedge d=c$. Then, by compatibility of $T$, $(a,d)=(b\vee d,c\vee d)\in T$, from which we obtain 
$$(a,c)=(a\wedge a,b\wedge d)\in T.$$ 
This shows transitivity of $T$ and finishes the proof.
\end{proof}

Recall that a lattice $L$ is {\it modular} if it fulfills the modular quasi-identity: for all $a,b,c\in L$, $a\leq c$ yields 
$a\vee (b\wedge c)=(a\vee b)\wedge c$. It is well known that modular complemented lattices are relatively complemented, cf. \cite{G1}. Hence, we obtain the following corollary:

\begin{corollary}
If $L$ is finite, simple, modular, and complemented lattice, then $L\in \mathsf{S}_{\mathsf{agg}}$. 
\end{corollary}  

One can easily see that the lattice depicted on Figure \ref{fig1} is simple, modular and complemented. The following theorem gives another sufficient condition for a lattice $L$ to has a smallest set of aggregation functions, which is fulfilled by the lattice from Figure \ref{fig1} as well. On the other hand, the lattice from Figure \ref{fig2} does not satisfy neither the above sufficient condition (it is simple, non-modular, non-complemented) nor the condition mentined below. 

\begin{theorem}
Any finite simple lattice $L$ for which the join of atoms is $1$ or the meet of coatoms is $0$ belongs to $\mathsf{S}_{\mathsf{agg}}$. 
\end{theorem}

\begin{proof}
Assume that the join of atoms of $L$ is $1$. Further, let $T\neq id_L$ be a tolerance on $L$. For any natural number $n\in\mathbb N$ define inductively $T^1=T$ and $T^{n+1}=T^{n}\circ T$. It is well known that the relation $T_t=\bigcup \{T^n;\, n\in \mathbb N\}$, the so-called transitive closure of $T$, is the least transitive relation on $L$ containing $T$. Clearly, $T_t$ is a congruence on $L$, and as $L$ is simple, we conclude $T_t=L^2$. Consequently, 
$(0,1)\in T^n$ for some $n\in \mathbb N$ and thus there are $a_0,\dots,a_{n}\in L$, where $a_0=0, a_n=1$ and 
$(a_k,a_{k+1})\in T$ for all $k\in \{0,\dots,n-1\}$. We may assume that $a_k\leq a_{k+1}$ for all $k\in \{0,\dots,n-1\}$.
Indeed, by compatibility of $T$, $(a_0,a_1)\in T$ yields $(a_0,a_0\vee a_1)\in T$. This together with $(a_1,a_2)\in T$ gives 
$(a_0\vee a_1,a_0\vee a_1\vee a_2)\in T$. In the same way we obtain 
$(a_0\vee a_1\vee \dots \vee a_k,a_0\vee a_1\vee \dots \vee a_{k+1})\in T$ for all $k\in \{0,\dots,n-1\}$, where $a_0\vee a_1\vee \dots \vee a_n=1$.

Further, we may assume by Lemma \ref{lemn3} (ii) that $a_1$ is an atom of $L$ and, analogously, $a_{n-1}$ is its coatom. 

We will show that $(0,p)\in T$ for any atom $p\in L$. Evidently, this property holds for $p=a_1$. 
Further, let $p\neq a_1$ be an arbitrary atom of $L$. If there is $a_k$ with $p\not\leq a_k$ and $p\leq a_{k+1}$, by compatibility 
of $T$ we obtain $(0,p)=(p\wedge a_{k},p\wedge a_{k+1})\in T$. We will show that for $p$ such an element $a_k$ always exists.
Indeed, we have $p\leq a_n=1$, thus if $p\not\leq a_{n-1}$, we are done. In the opposite case we have $p\leq a_{n-1}$, hence again, $p\not\leq a_{n-2}$ verifies the existence of the desired element. Clearly, as $p\not\leq a_1$, the same reasoning leads after finitely many steps to the result. 

Consequently, we have proved $(0,p)\in T$ for any atom $p\in L$. Finally, as the join of atoms equals to $1$, applying compatibility of $T$, the last property yields 
$$(0,1)=(0,\bigvee \{p;\, p\in \mathsf{At}(L)\})\in T,$$ 
and hence $T=L^2$. Altogether we have $\mathsf{Tol}(L)=\{id_L,L^2\}$. 

The proof of the dual case when the meet of coatoms of $L$ is $0$ can be done in a similar way.
\end{proof}

\section{Conclusion}
In this paper we have shown that finite lattices $L$ for which the set $\mathsf{C}(L)$ of aggregation functions coincides with the set $\mathsf{Pol}_{0,1}(L)$ of its polynomials can be completely characterized by their tolerances. Moreover, we have mentioned several lattice-theoretical conditions which are sufficient for this property. 
We believe that our results can be used also for an analysis of special classes of aggregation 
functions on lattices or even certain posets. 

In the future work we would like to extend this idea to the study of lattices admitting certain richer classes of aggregation functions.

\section{Acknowledgments}
The authors would like to thank the anonymous reviewers for their helpful and constructive comments which helped enhance the presentation of this paper.

The second author was supported by the ESF Fund CZ.1.07/2.3.00/30.0041 and by the Slovak VEGA Grant 2/0028/13, the first author by the international project Austrian Science Fund (FWF)-Grant Agency of the Czech Republic (GA\v{C}R) I 1923-N25, by the AKTION project "Ordered structures for Algebraic Logic" 71p3 and by the Palack\'{y} University project IGA PrF 2015010.

%% main text

\end{document}